\documentclass[12pt]{amsart}
\usepackage{epsfig,color}
\usepackage{blindtext}
\usepackage{hyperref}
\usepackage{graphicx}
\usepackage{enumitem} 
\usepackage{url}
\usepackage{amssymb}
\usepackage{graphicx,import}
\usepackage{comment}

\theoremstyle{definition}

\newtheorem{theorem}{Theorem}[section]
\newtheorem{definition}[theorem]{Definition}

\newtheorem{lemma}[theorem]{Lemma}
\newtheorem{remark}[theorem]{Remark}
\newtheorem{corollary}[theorem]{Corollary}

\newtheorem*{remark*}{Remark}

\numberwithin{equation}{section}

\newcommand{\dv}{\text{div}}

\newcommand{\mc}{\mathcal}
\newcommand{\mb}{\mathbb}

\newcommand{\eps}{\varepsilon}

\newcommand{\R}{\mathbb{R}}

\title[Singularities of Mean Curvature Flow of Surfaces with Additional Forces]{Singularities of Mean Curvature Flow of Surfaces with Additional Forces}

\date{\today}

\author{Ao Sun}
\address{Department of Mathematics, Massachusetts Institute of Technology, Cambridge, MA 02139, USA}
\email{aosun@mit.edu}

\begin{document}

\begin{abstract}
In this paper we study the blow up sequence of mean curvature flow of surfaces in $\mb R^3$ with additional forces. We prove that the blow up limit of a mean curvature flow of smoothly embedded surfaces with additional forces with finite entropy is a smoothly embedded self-shrinker.
\end{abstract}

\maketitle

\section{Introduction}
We say a family of submanifolds $\{M_t\}_{t\in[T_0,T_1]}$ in $\mb R^N$ is a \emph{mean curvature flow} if this family satisfies the equation
\begin{equation}
\frac{\partial x}{\partial t}=\vec H(x),
\end{equation} 
where $\vec H$ is the mean curvature vector defined to be the trace of the second fundamental forms over $M_t$. We say $\{M_t\}_{t\in[T_0,T_1]}$ is a \emph{mean curvature flow with additional forces} if this family satisfies the equation 
\begin{equation}
\frac{\partial x}{\partial t}=\vec H(x)+\beta(x,T_xM),
\end{equation} 
where $\beta$ is the \emph{Brakke operator} introduced by White in \cite{white2005local}, and it plays the role as the additional forces. For example, $\beta$ can be a vector field $V(x)$ defined on $\mb R^N$, and then $\{M_t\}$ flows by mean curvature together with the external force field $V(x)$.

It is a well-known fact that a mean curvature flow starting from a closed surface $M_{T_0}$ must develop singularities in finite time. We use a blow up method to generate the models of the singularities. More precisely, suppose $\{M_t\}$ is a mean curvature flow which becomes singular at $(y,s)\in\mb R^N\times\mb R$ in spacetime, then we can consider the parabolic rescaling sequence (also called \emph{blow up sequence}) $M_t^{\lambda_i}$ for $\lambda_i\to 0$:
\[M_t^{\lambda_i}:=\lambda_i^{-1}(M_{s+\lambda^2 t}-y).\]
Using Huisken's monotonicity formula (see \cite{huisken1990asymptotic}), Ilmanen \cite{ilmanen1995singularities} and White \cite{white94partial} showed that this blow up sequence weakly converges to a \emph{Brakke flow} $\{\nu_t\}_{t<0}$, which should be understood as a weak mean curvature flow. Moreover, the Brakke flow is \emph{self similar} in the sense that $\nu_t$ is a rescaling of $\nu_{-1}$, and the support of $\nu_{-1}$ satisfies an elliptic equation
\[\vec H(x)+\frac{x^\bot}{2}=0.\]
This limit is also called the \emph{tangent flow} of $\{M_t\}$ at $(y,s)$. The next natural question is to study the regularity of the blow up limit $\nu_t$, in particular the time slice $\nu_{-1}$. In this paper, we will follow Ilmanen's 1995 preprint \cite{ilmanen1995singularities} to prove the following regularity result for mean curvature flow with additional forces:
\begin{theorem}\label{T:Main Theorem}
Let $\{M_t\}_{t\in[0,s)}$ be a mean curvature flow in $\R^3$ with additional forces with bounded $L^\infty$ norm. Suppose $M_0$ is a smoothly embedded surface with finite genus and finite entropy, and $\{\nu_t\}$ is a blow up limit at $(y,s)$. Then the support of $\nu_t$ is smoothly embedded, and satisfies the self-shrinker equation
\begin{equation}
\vec H+\frac{x^\bot}{-2t}=0,\quad t<0.
\end{equation}
\end{theorem}

When $\beta\equiv0$, we recover the main theorem of \cite{ilmanen1995singularities}.

In Section \ref{S:Mean Curvature Flow with Additional Forces} we will introduce the notion of entropy, which is a very important quantity in the study of mean curvature flow used by Colding-Minicozzi in \cite{colding2012generic}.

The proof of Theorem \ref{T:Main Theorem} works just for surfaces in $\R^3$. On one hand, the proof relies on many tools which can only be applied to surfaces, for example the local Gauss-Bonnet theorem and Simon's graph decomposition theorem. For higher dimensional submanifold, the generalization of these tools haven't been developed yet; On the other hand, in \cite{ilmanen1995singularities} Ilmanen constructed examples to show that the blow up limit of mean curvature flow of embedded surfaces can be a surface which is not smoothly embedded. Also, the failure of classical regularity theory of submanifolds with higher codimension also implies that Theorem \ref{T:Main Theorem} may not be true for higher codimension.

Compared with Ilmanen's proof in \cite{ilmanen1995singularities} for mean curvature flow, new technical issues arise in the proof for mean curvature flow with additional forces. The main technical issue is that many geometric quantities are no longer monotone decreasing. So we need to carefully control the growth of these geometric quantities, for example Lemma \ref{L:local area bound of MCF with additional force}. We also need to prove that the influence of the additional forces is negligible after blow up.

\vspace{4pt}
\subsection{Mean Curvature Flow}
Mean curvature flow is the negative gradient flow of the volume functional of the submanifolds. It was first studied in material science in 1950s to model the motion of surfaces such as bubbles, cell membranes, the interface of two different materials, etc., but the modern study of mean curvature flow only began in 1970s. See \cite{colding-minicozzi-pedersen} for more historical background. Currently there are several approaches to study mean curvature flow.

The first approach is the measure theory approach, first studied by Brakke in his seminal book \cite{brakke2015motion}. Brakke defined the measure theoretic mean curvature flow, which is now called \emph{Brakke flow}. Brakke studied the regularity and singular behavior of Brakke flows. In particular, Brakke proved a regularity theorem, which can be viewed as the flow version of Allard's regularity theorem for integral rectifiable Radon measure in the study of minimal surfaces. Later in \cite{ilmanen1994elliptic} Ilmanen use the method of elliptic regularization to construct a weakly continuous family of Brakke flow. This approach plays an important role in our paper, especially in Section \ref{S:Weak Blow Up Limit} we will follow Brakke's approach and the generalization by Ilmanen in \cite{ilmanen1994elliptic} to prove the existence of a weak blow up limit.

The second approach uses the parametric method of differential geometry which treat mean curvature flow as a parabolic geometric equation and uses the tools in partial differential equation to study mean curvature flow. The first remarkable result of this approach, first proven by Huisken in \cite{huiksen1984convex}, was that any mean curvature flow which starts from a convex hypersurface will finally shrink to a single singular point, and the blow up limit at the singular point is a sphere. Later Huisken-Sinestrari \cite{huiksen-sinestrari} also uses this approach to study the mean curvature flow starting from a mean convex hypersurface. See \cite{ecker2012regularity} for an introduction to this approach.

The third approach is the level set flow approach. This approach was first introduced by \cite{osher-sethian} in numerical analysis, and later justified by Chen-Giga-Goto \cite{chen-giga-goto} and Evans-Spruck \cite{evans-spruck} independently. Based on their work, Ilmanen introduced a so-called ``biggest flow" in \cite{ilmanen1992generalized}, which played an important role in the study of mean curvature flow of mean convex sets. For example see White \cite{white2000meanconvex}.

If a mean curvature flow starts from a smooth compact hypersurface, then all the approaches above give us the same flow before the first singular time.


Every approach described above has their advantages in the study of different problems of mean curvature flow. In this paper we will mainly use the techniques from the first approach and the second approach.

\subsection{Singularities of Mean Curvature Flow}
The singularities of mean curvature flow in general are unavoidable, and they are characterized by special ancient homothetic solutions which are self-similar. Since these solutions are homothetic solutions, it is uniquely determined by a single time slice $\Sigma_{-1}$. $\Sigma_{-1}$ is called \emph{self-shrinker} and satisfies the self-shrinker equation
\[\vec H+\frac{x^\bot}{-2}=0,\]
which is an elliptic system. An important observation is that this equation is actually the minimal surface equation in the Gaussian space $(\mb R^N,e^{-|x|^2/(2N)}\delta_{ij})$. So self-shrinkers are actually minimal varieties in the Gaussian space.

The simplest self-shrinkers are the Euclidean subspace $\mb R^n$ and the sphere of radius $\sqrt{-2n}$. The latter characterizes the singularity generated in the mean curvature flow starting from a convex hypersurface as shown by Huisken in \cite{huiksen1984convex}. Later in \cite{huisken90}, \cite{huisken93} Huisken studied mean convex self-shrinkers and showed that under mild assumptions, the only mean convex self-shrinkers are the hyperplane, the spheres and the generalized cylinders $\mb S^k\times\mb R^{n-k}$. Later in \cite{colding2012generic} Colding-Minicozzi removed those assumptions.

Removing the mean convex assumption allowed for fruitful examples of self-shrinkers. In \cite{angenent1992}, Angenent constructed a genus one example of self-shrinkers called the ``Angenent donut". There are also numerical examples, cf. \cite{chopp1994computation}. More recently Kapouleas-Kleene-M\o ller \cite{kapouleas-etc} constructed self-shrinkers of high genus. However, the classification results of self-shrinkers seem far from complete. The best result so far, proven by Brendle in \cite{brendle-genus0}, proved that the only genus zero self-shrinker in $\mb R^3$ is the sphere.

The next question is how the hypersurfaces' motion by mean curvature approaches the singularity. This can be tackled by studying the convergence behavior of the blow up sequence of mean curvature flow at the singularity. The first result, as we mentioned before, proven by Huisken in \cite{huiksen1984convex}, was that the blow up sequence of convex hypersurfaces converges to a sphere smoothly. Later Huisken introduced the so-called \emph{Type I} hypothesis on the flow and this hypothesis implies a curvature bound for the blow up sequence, hence implies $C^2$ convergence of the blow up sequence to a smooth self-shrinker.

However, the Type I hypothesis is difficult to verify, so people try to remove the Type I hypothesis. The most successful result is obtained by White \cite{white2000meanconvex} for mean curvature flow of compact mean convex hypersurfaces. White proved that the blow up sequence will converge smoothly to a hyperplane, or a sphere or a cylinder.

Without any curvature assumptions, no such strongly convergence result was known to the blow up sequence of mean curvature flow. However in the weak sense of Brakke \cite{brakke2015motion} (see also \cite{white94partial} and \cite{ilmanen1995singularities}) the blow up sequence converges to a self-similar homothetic flow. The natural question is the regularity of this weak limit. In a preprint \cite{ilmanen1995singularities} Ilmanen proved that for a mean curvature flow of smoothly embedded surface in $\mb R^3$, the limit is a smoothly embedded self-shrinker. He also pointed out in the same preprint that the limit may not be smooth in higher dimensions.

Ilmanen also conjectured that the convergence is of multiplicity $1$. Note that once we know the convergence is of multiplicity $1$, then the convergence should be smooth. This seems like the key part of the understanding of mean curvature flow of surfaces in $\mb R^3$. Note according to White's result, for mean curvature flow of mean convex hypersurfaces, the convergence is of multiplicity $1$. 

\subsection{Additional Forces}\label{SS:Additional Forces}
Mean curvature flow with additional forces is a natural generalization of mean curvature flow. Let us see some examples (see also \cite{white2005local}). We will use $\beta$ to denote the additional forces and give the explicit definition in Section \ref{S:Mean Curvature Flow with Additional Forces}.

\begin{enumerate}
\item Let $\beta$ be a constant vector field over $\mb R^3$ pointing to the negative direction of the $z$ axis. Then this mean curvature flow with additional forces is the model of all the physical phenomena we mentioned before with gravity.

\item Let $\{M_t\}$ be a family of hypersurfaces flowed by mean curvature with additional forces \[\beta=\frac{\int_{M_t}Hd\mu}{\int_{M_t}1d\mu}.\]
This flow is also called the \emph{volume preserving mean curvature flow}, which was first studied by Huisken in \cite{huisken87volume}. The static solutions are constant mean curvature hypersurfaces.

\item Suppose $S$ is a compact submanifold in $\mb R^N$. Let $\{M_t\}$ be a family of hypersurfaces of $S$ flowed by mean curvature with additional forces with $\beta$ equal minus the trace of the second fundamental form of $S$ restricted to the tangent space of $M_t$. Then $\{M_t\}$ is a mean curvature flow in $S$. In particular, for any $N-1$ dimensional compact Riemannian manifold $S$, if we can isometrically smoothly embedded it into $\mb R^N$, then any mean curvature flow in $S$ can be interpret to be a mean curvature flow with additional forces in $\mb R^N$.

\item \emph{Rescaled Mean Curvature Flow}. Let 
\[\beta=\frac{x^\bot}{2}.\]
A family of hypersurfaces that satisfies the mean curvature flow with this additional forces is called the rescaled mean curvature flow. A rescaled mean curvature flow is equivalent to a mean curvature flow up to a rescaling of space and a reparametrization of time. In fact suppose $\Sigma_t$ is a mean curvature flow, then $e^{-t/2}\Sigma_{-e^{-t}}$ is a rescaled mean curvature flow and vice versa. See \cite[Section 2]{colding-ilmanen-minicozzi-white}.

Rescaled mean curvature flow is a very powerful tool in the study of self-shrinkers. See \cite{colding-ilmanen-minicozzi-white} and \cite{hershkovits2018sharp} for more details. In particular, the singularities of the rescaled mean curvature flows reflect the singular behavior of the mean curvature flows.
\end{enumerate}

There is much research about the mean curvature flow with additional forces, for example see \cite{white2005local}. 

The singular behavior of mean curvature flow and mean curvature flow with additional forces are similar. The main reason is that after blowing the flow the additional forces become negligible. So the blow up sequence of mean curvature flow with additional forces turns to be similar to the blow up sequence of mean curvature flow. 

\subsection{Idea of the Proof}
It would be heuristic to compare the proof to the compactness theorem of embedded minimal surfaces in three dimensional manifold, see \cite{choi1985space}, \cite{white1987curvature} and \cite{colding2012smooth}. For example, in \cite{choi1985space}, in order to prove a subsequence converges to a minimal surface, Choi-Schoen first derived a curvature estimate to show that at the points where no curvature concentration happens, there is a subsequence converging to a smooth minimal surface nicely. At those points where curvature concentration happens, the removable of singularity theorem can be applied. Then they obtained a smooth limit surface.

With the similar argument, Ecker \cite[Section 5]{ecker2012regularity} proved that if the density of a mean curvature flow is not much larger than $1$, then there would be a curvature estimate like Choi-Schoen. As a result the flow is actually regular at that point. However without the density assumption, we can not get a curvature estimate like Choi-Schoen. One reason is that the total curvature is not scale-invariant under parabolic rescaling (compare to the minimal surface problem, the total curvature is scale-invariant under Euclidean rescaling).

To overcome this issue, we need an argument to help us finding a convergence subsequence without a curvature estimate. The first ingredient is a scale-invariant total curvature estimate. With the uniform bound of entropy, we can get a scale-invariant integral mean curvature bound. Then a generalized Gauss-Bonnet Theorem \cite[Theorem 3]{ilmanen1995singularities} (see also Section \ref{S:Tools for Analysis of Surfaces}) can change it to a scale-invariant integral curvature bound, which allows us to find time slices in the blow up sequence with total curvature bound.

Then at the points without curvature concentration, if the multiplicity of the sequence of time slices in the blow up sequence is $1$, we may pass it to a subsequence with a measure theoretical limit, and then Allard's regularity Theorem \ref{T: Allard's regularity theorem} implies the limit is a smooth self-shrinker. In general the multiplicity is not $1$, so we need a technical Lemma by Simon (Theorem \ref{T:Simon's Graph Decomposition Theorem}) that if there is no curvature concentration, then we may decompose the surface locally into several sheets, and each sheet has multiplicity $1$. Then we may run the argument above again to conclude that the limit is a smooth self-shrinker (with multiplicities).

Finally we study the points where the curvature concentrate. We can show they are discrete, then the removable of singularities theorem can be applied to show that the limit is smoothly embedded everywhere.

\subsection{Organization of the Paper}
In Section \ref{S:Mean Curvature Flow with Additional Forces}, we introduce the definition of mean curvature flow with additional forces, under the parabolic equation setting and the measure theoretic setting. Then we prove the important monotonicity formula of mean curvature flow with additional forces. At the end of this section, we study the flow under parabolic rescaling.

In Section \ref{S:Weak Blow Up Limit}, we follow \cite{ilmanen1994elliptic}, \cite{ilmanen1995singularities} to prove that the weak blow up limit of mean curvature flow with additional forces is a self similar Brakke flow.

In Section \ref{S:Tools for Analysis of Surfaces}, we introduce the tools which we use to prove the main Theorem \ref{T:Main Theorem}: Ilmanen's generalized Gauss-Bonnet theorem, Allard's regularity theorem and Simon's graph decomposition theorem.

In Section \ref{S:Smoothness of Blow Up Limit}, we prove the main theorem.

And in the Appendix we have the technical lemma used.
\subsection{Acknowledgement}
The author wants to thank Professor Bill Minicozzi for his advisory and comments, and his encouragement.

\section{Mean Curvature Flow with Additional Forces}\label{S:Mean Curvature Flow with Additional Forces}
In this section we will follow the idea in \cite{white2005local} to generalize mean curvature flow to mean curvature flow with additional forces. The discussion in this section holds for arbitrary dimensional mean curvature flows in arbitrary dimensional ambient space $\mb R^{N}$.

In order to define the additional forces, we follow \cite{white2005local} to introduce the {Brakke operator}.

\begin{definition}
A {\it Brakke operator} is a map 
\[\beta:\mbox{a subset of $(\mb R^{N+1}\times G(m,N))$}\to\mb R^N.\]
Here $\mb R^{N+1}$ is the spacetime and $G(m,N)$ is the Grassmannian of $m$-planes in $\mb R^N$. We use $G^{(m,N)}$ to denote the trivial Grassmannian bundle $(\mb R^{N+1}\times G(m,N))$.
\end{definition}

We define the mean curvature flow with additional forces $\beta$ to be a family of $m$-dimensional submanifolds $\{M_t\}_{t\in[T_0,T_1)}$ smoothly immersed in $\mb R^N$ satisfy the equation:

\begin{equation}\label{E:MCF with additional forces}
\frac{\partial x}{\partial t}=\vec{H}(x)+\beta(x,T_{x}M_t),~x\in M_t,~t\in[T_0,T_1).
\end{equation} 

Here $\vec{H}$ is the mean curvature vector of $M_t$, $T_{x}M_t\in G(m,N)$ is the tangent space of of $M_t$ at $x$.  

From now on, we will only consider the additional forces with bounded $L^\infty$ norm. Note that the examples (1), (3) in Section \ref{SS:Additional Forces} has the additional forces with bounded $L^\infty$ norm, and the example (4) has the additional forces with bounded $L^\infty$ norm if the flow is uniformly bounded.

\subsection{Varifold and Brakke Flow}
In this subsection, we generalize the mean curvature flow with additional forces to measure theoretic setting. In mean curvature flow case, this measure theoretic generalization is known as Brakke flow. We will call our generalization {\it Brakke flow (with additional forces)}. For introduction to geometric measure theory, see \cite{simon1983lecture}, and \cite{brakke2015motion}, \cite{ilmanen1994elliptic} especially for mean curvature flow.

The readers who are not familiar with geometric measure theory can just skip the geometric measure theoretic settings in this section, and treat the varifold just be the submanifold $M_t$, and treat the integral with respect to the measure $\mu_t$ is just the integral over the submanifold $M_t$. The skip of the geometric measure theoretic settings would not affect the understanding the essential ingredients of the proof of the main theorem in Section \ref{S:Smoothness of Blow Up Limit}.

Let $\{\mu_t\}_{t\in[T_0,T_1)}$ be a family of $m$-rectifiable Radon measures ($T_1$ can be $\infty$). We define $S=S_\mu:\mb R^N\to G(m,N)$ to be the operator maps $x$ to its tangent space, where the tangent space $T_x\mu$ should be understood as the tangent space for Radon measure. We will also use $\top,\bot$ as the superscript of a vector to denote the projection to $T_x\mu$ and $(T_x\mu)^\bot$ respectively. We say $\mu_t$ form a Brakke flow with additional forces $\beta$ if they satisfy the equation \ref{E:MCF with additional forces} in the weak form

\begin{equation}
\int\phi d\mu_{t_2}\leq \int\phi d\mu_{t_1}+\int_{t_1}^{t_2}\int (\vec{H}+\beta(x,S(x)))\cdot(-\phi \vec H+ D^\bot\phi)d\mu_t dt
\end{equation}

for all $T_0\leq t_1\leq t_2 < T_1$ and nonnegative test function $\phi\in C_c^1(\mb R^N)$. 

Similar to Brakke's criterion in \cite{brakke2015motion}, we have the following weak form of Brakke flow with additional forces with test function $\phi\in C_c^1( R^N\times[T_0,T_1)\to\mb R)$ depending on time (See also \cite{ecker2012regularity}):

\begin{equation}\label{E:Weak form of Brakke Flow}
\int\phi(\cdot,t_2) d\mu_{t_2}\leq \int\phi(\cdot,t_1) d\mu_{t_1}+\int_{t_1}^{t_2}\int (\vec{H}+\beta(x,S(x)))\cdot(-\phi \vec H+D^\bot\phi)+\frac{\partial \phi}{\partial t}d\mu_t dt.
\end{equation}

When $\mu_t$'s are the Radon measures of smooth embedded submanifolds in $\mb R^N$, this weak form definition is equivalent to the mean curvature flow defined by equation (\ref{E:MCF with additional forces}).

\subsection{Entropy and Area Growth Bound}
The entropy of mean curvature flow is introduced by Colding-Minicozzi \cite{colding2012generic}, which plays an important role in the study of the mean curvature flow. Given a Radon measure $\mu$, we define the functional $F_{y,s}$ to be 
\begin{equation}
F_{y,s}(\mu)=\frac{1}{(4\pi s)^{m/2}}\int e^{\frac{-|x-y|^2}{4s}}d\mu,
\end{equation}
and define the \emph{entropy} $\lambda$ of $\mu$ to be the supremum among all possible $y\in\R^N,s\in(0,\infty)$:

\begin{equation}
\lambda(\mu)=\sup_{(y,s)\in\mb R^N\times(0,\infty)}F_{y,s}(\mu).
\end{equation}

Entropy characterizes the submanifold from all scales. One important fact is that the entropy bound and the area growth bound is equivalent, in the following sense:

\begin{theorem}\label{T:area and entropy are equivalent}
\begin{equation}
\sup_{x\in\mb R^N}\sup_{R>0}\frac{\mu(B_R(x))}{R^m}\leq C\lambda(\mu),
\end{equation}
\begin{equation}
\lambda(\mu) \leq C \sup_{x\in\mb R^N}\sup_{R>0}\frac{\mu(B_R(x))}{R^m},
\end{equation}
where $C$ is a universal constant.
\end{theorem}

\begin{proof}
On one hand, for any $y\in\mb R^N$ and $R>0$, we have
\[\lambda(\Sigma)\geq \frac{1}{(4\pi t)^{m/2}}\int e^{\frac{-|x-y|^2}{4t}}\chi_{B_R(y)}d\mu \geq \frac{1}{(4\pi t)^{m/2}}e^{-R^2/4t}\mu(B_R(y)).\]
Then choose $t=R^2$ we have 
\[\frac{\mu(B_R(y))}{R^m}\leq C\lambda(\Sigma),\]
where $C$ is a universal constant. So we conclude that the first inequality holds.

On the other hand, since the statement is translation and scaling invariant, we only need to prove $\int e^{-|x|^2}d\mu\leq C \sup_{x\in\mb R^N}\sup_{R>0}\frac{\mu(B_R(x))}{R^m}$ to conclude the second statement. 
\begin{equation}
\begin{split}
\int e^{-|x|^2}d\mu&\leq\sum_{y\in\mb Z^N}\int e^{-|x|^2}\chi_{B_2(y)}d\mu\\
&\leq C\sum_{y\in\mb Z^N} e^{-|y|^2}\mu(B_2(y))\leq C\sum_{y\in\mb Z^N}e^{-|y|^2}\mu(B_2(y))\\
&\leq C \sup_{x\in\mb R^N}\sup_{R>0}\frac{\mu(B_R(x))}{R^m}.
\end{split}
\end{equation}
Here $C$ varies from line to line, and we use the second statement in the second inequality. Then we conclude this theorem.
\end{proof}

From now on, we will assume our initial slice $\mu_{T_0}$ of the Brakke flow with additional forces has finite entropy. This is a natural assumption. For example, in \cite[Lemma 7.2]{colding2012generic} Colding-Minicozzi proved that a closed hypersurface has finite entropy.

\subsection{Monotonicity Formula}
In this section we study some quantitative behaviours of $F_{y,s}$ functional and entropy. 

When we set the test function $\phi$ in (\ref{E:Weak form of Brakke Flow}) to be the heat kernel, we will get the monotonicity formula generalized the monotonicity formula of mean curvature flow. Let 
\begin{equation}
\rho_{y,s}(x,t)=\frac{1}{(4\pi(s-t))^{m/2}}e^{\frac{-\vert x-y\vert^2}{4(s-t)}}
\end{equation}
be the heat kernel. It is easy to compute that:
\begin{equation*}
\begin{split}
&D\rho_{t,s}=-\frac{(x-y)}{2(s-t)}\rho,\\
&\frac{\partial \rho_{y,s}}{\partial t}=\frac{m}{2(s-t)}\rho+\frac{-\vert x-y\vert^2}{4(s-t)^2}\rho,\\
&\Delta_{\mu_t}\rho_{y,s}=\dv_{\mu_t}D\rho=-\frac{m}{2(s-t)}\rho+\frac{\vert (x-y)^{\top}\vert^2}{4(s-t)}\rho.
\end{split}
\end{equation*}

\begin{theorem}\label{T:monotonicity formula}[Monotonicity Formula for Brakke Flows with Additional Forces, see \cite[Theorem 3.1]{huisken1990asymptotic}, \cite[Lemma 7]{ilmanen1995singularities}, \cite[Theorem 4.11]{ecker2012regularity}]
For any Brakke flow of additional forces $\{\mu_t\}$ satisfies (\ref{E:Weak form of Brakke Flow}), any $(y,s)\in\mb R^N\times [T_0,\infty)$ and all $T_0\leq t_1\leq t_2<\min\{s,T_1\}$, we have
\begin{equation}
\begin{split}
&\int\rho(x,t_2)d\mu_{t_2}+\int_{t_1}^{t_2}\int\rho(x,t)\left\vert \vec H+\frac{(x-y)^{\bot}}{2(s-t)}-\frac{\beta}{2}\right\vert^2d\mu_t dt\\
&\leq\int\rho(x,t_1) d\mu_{t_1}+\int_{t_1}^{t_2}\int\frac{\vert \beta\vert^2}{4}\rho(x,t) d\mu_{t}dt
\end{split} 
\end{equation}
\end{theorem}

Before we prove this theorem, let us first prove some bounds for other important geometric quantities. The following local area bound is a natural generalization of \cite[Proposition 4.9]{ecker2012regularity}. 

\begin{lemma}\label{L:local area bound of MCF with additional force}
Given $\rho>0$ and suppose $[t_0-r^2/(8m),t_0]\subset[T_0,T_1]$, then we have
\[\mu_t(B_{r/2}(x_0))\leq 8 e^{(C+C/r)(t-(t_0-\frac{r^2}{8m}))} \mu_{t_0-\frac{r^2}{8m}}(B_r(x_0)).\]
\end{lemma}

\begin{proof}
Let us define a test function appears in \cite{brakke2015motion}. For $r>0$, $(x_0,t_0)\in\mb R^N\times\mb R$ we define
\begin{equation}
\varphi_{(x_0,t_0),r}(x,t)=\left(1-\frac{\vert x-x_0\vert^2+2m(t-t_0)}{r^2}\right)^3_+,
\end{equation}
where $+$ denotes the positive part of the function. By \cite[Page 51,(4.4)]{ecker2012regularity}, we have
\[\left(\frac{d}{dt}-\Delta_{\mu_t}\right)\varphi_{(x_0,t_0),r}(x,t)\leq0.\]
Now we fix $r,x_0,t_0$ and use simple notation $\varphi$ to denote $\varphi_{(x_0,t_0),r}$. For $t\in[t_0-\frac{r^2}{8m},t_0]$, insert this test function into (\ref{E:Weak form of Brakke Flow}) with $[t_1,t_2]=[t_0-\frac{r^2}{8m},t]$, we get
\begin{equation}\label{E1: equation in the area bound 1}
\int \varphi d\mu_t\leq \int \varphi  d\mu_{t_0-\frac{r^2}{8m}}+\int_{t_0-\frac{r^2}{8m}}^t\int -\varphi |\vec H|^2+ \beta \cdot (-\varphi  \vec H+D^\bot \varphi) d\mu_s ds.
\end{equation}

Note $|\beta ||\varphi  \vec H|\leq |\beta|^2\varphi/4+|\vec H|^2\varphi$ and $\beta \cdot D^\bot \varphi\leq |D^\bot \varphi| |\beta|$, we get

\[\int \varphi d\mu_t\leq \int \varphi  d\mu_{t_0-\frac{r^2}{8m}}+C\int_{t_0-\frac{r^2}{8m}}^t\int |\beta|^2\varphi+ |D^\bot \varphi| |\beta| d\mu_s ds.\]

Observe that for $|x-x_0|^2\leq r^2/4$ and $2m(t-t_0)\leq r^2/4$ we have $\varphi_{(x_0,t_0),r}(x,t)\geq 1/8$ and $|D^\bot \varphi| \leq C\varphi/r $ where $C$ is a universal constant. Then by the boundedness of $\beta$ we get  

\[\int \varphi d\mu_t\leq \int \varphi  d\mu_{t_0-\frac{r^2}{8m}}+(\frac{C}{r}+C)\int_{t_0-\frac{r^2}{8m}}^t\int \varphi d\mu_s ds.\]
Then by the comparison Lemma \ref{L:Comparison Lemma for Integration} we get
\[\int \varphi d\mu_t\leq e^{(C+C/r)(t-(t_0-\frac{r^2}{8m}))}\int \varphi  d\mu_{t_0-\frac{r^2}{8m}},\]
which gives
\[\mu_t(B_{r/2}(x_0))\leq 8 e^{(C+C/r)(t-(t_0-\frac{r^2}{8m}))} \mu_{t_0-\frac{r^2}{8m}}(B_r(x_0)).\]
\end{proof}

\begin{corollary}\label{Cor: Finite gaussian}
If $\int\rho_{y,s}(x,T_0) d\mu_{T_0}\leq\infty$ then $\int\rho_{y,s}(x,t) d\mu_{t}\leq\infty$ for all $t<s$.
\end{corollary}

\begin{proof}
Given $t<s$ we have
\begin{equation}\label{Cor E:1}
\frac{\rho_{y,s}(x,t)}{\rho_{y,s}(x,T_0)}=\frac{(s-T_0)^m}{(s-t)^m}e^{-\frac{|x-y|^2(t-T_0)}{4(s-t)(s-T_0)}}\leq C,
\end{equation}
where $C$ is a universal constant does not depend on $x,y$. Now we choose $r_0>0$ such that $T_0+r_0^2/(2m)=t$. Then for $x'$ with $|x'-x|\leq r_0$ we have
\begin{equation}\label{Cor E:2}
\frac{\rho_{y,s}(x',T_0)}{\rho_{y,s}(x,T_0)}=e^{-\frac{1}{4(s-t)}(|x-y|^2-|x-y'|^2)}\in(\frac{1}{C(r_0)},C(r_0)),
\end{equation}
where $C(r_0)$ does not depend on $x,y$.

Lemma \ref{L:local area bound of MCF with additional force} gives
\[\mu_t(B_{r_0/2}(x))\leq C \mu_{T_0}(B_{r_0}(x))\]
for all $x\in\mb R^N$, where $C$ is a constant does not depend on the choice of $x$. Let us consider the balls 
\[\mc B=\{B_{r_0/2}(x):x=(a_1(r_0/2),a_2(r_0/2),\cdots,a_N(r_0/2)),a_1,\cdots,a_N\in\mb Z\}.\]
These balls cover whole $\mb R^N$. Then 
\[\tilde{\mc B}=\{B_{r_0}(x):x=(a_1(r_0/2),a_2(r_0/2),\cdots,a_N(r_0/2)),a_1,\cdots,a_N\in\mb Z\}\]
also cover whole $\mb R^N$, and each point in $\mb R^N$ is covered by at most $C(N)$ number of balls in $\tilde{\mc B}$. Thus we can estimate
\begin{equation}
\begin{split}
\int\rho_{y,s}(x,t) d\mu_{t}&\leq\sum_{B_{r_0/2}(z)\in\mc B}\int\rho_{y,s}(x,t)\chi_{B_{r_0/2}(z)} d\mu_{t}\\
\text{By (\ref{Cor E:1})}&\leq C \sum_{B_{r_0/2}(z)\in {\mc B}}\int\rho_{y,s}(x,T_0)\chi_{B_{r_0/2}(z)} d\mu_{t}\\
\text{By (\ref{Cor E:2})}&\leq C \sum_{B_{r_0/2}(z)\in {\mc B}}\rho_{y,s}(z,T_0)\int\chi_{B_{r_0/2}(z)} d\mu_{t}\\
\text{(By Lemma \ref{L:local area bound of MCF with additional force})} &\leq C \sum_{B_{r_0}(x)\in\tilde{\mc B}}\rho_{y,s}(z,T_0)\int\chi_{B_{r_0}(z)} d\mu_{T_0}\\
\text{By (\ref{Cor E:2})}&\leq C \sum_{B_{r_0}(z)\in \tilde{\mc B}}\int\rho_{y,s}(x,T_0)\chi_{B_{r_0}(z)} d\mu_{T_0}\\
&\leq C C(N)\int\rho_{y,s}(x,T_0) d\mu_{T_0}<\infty.
\end{split}
\end{equation}
Here in the last line we use a covering argument. Then we conclude this corollary.
\end{proof}

\begin{proof}[Proof of Theorem \ref{T:monotonicity formula}]
We follow the proof of \cite[Lemma 7]{ilmanen1995singularities}. First we observe that integration by parts gives us:
\begin{equation}
\begin{split}
\int-\phi\vert\vec H\vert^2+\vec H\cdot D^\bot\phi+\frac{\partial}{\partial t}\phi d\mu_t&=\int-\phi\vert\vec H\vert^2+2\vec H\cdot D^\bot\phi-\vec H\cdot D^\bot\phi+\frac{\partial}{\partial t}\phi d\mu_t\\
&= \int -\phi\left\vert \vec H-\frac{D^\bot \phi}{\phi}\right\vert^2+Q_S(\phi)d\mu_t,
\end{split}
\end{equation}
where
\[Q_S(\phi)=\frac{\vert D^\bot\phi\vert^2}{\phi}+\Delta_{\mu_t}\phi+\frac{\partial\phi}{\partial t}.\]
Note $Q_S(\rho)=0$ for any $m$-space $S$. In particular, this identity holds for the approximate tangent space. Now we rewrite (\ref{E:Weak form of Brakke Flow}) to be
\begin{equation}\label{E:Rewrite weak form of Brakke Flow}
\int\phi(\cdot,t_2) d\mu_{t_2}\leq \int\phi(\cdot,t_1) d\mu_{t_1}+\int_{t_1}^{t_2}\int -\phi\left\vert \vec H-\frac{D^\bot \phi}{\phi}-\frac{\beta}{2}\right\vert^2+Q_S(\phi) + \phi\frac{ |\beta |^2}{4} d\mu_tdt.
\end{equation}

Let $\psi=\psi_R$ be a cutoff function which is supported in $B_{2R}$, with value constant $1$ on $B_R$, and has the gradient bound
\[R\vert D\psi\vert+R^2\vert D^2\psi\vert\leq C.\]
Then we insert the test function $\psi\rho$ to get
\[\vert Q_S(\psi\rho)\vert=\vert\psi Q_S(\rho)+\rho Q_S(\psi)+2D\psi\cdot D\rho\vert\leq C(\frac{1}{R^2}+\frac{1}{s-t})1_{B_{2R}(y)\setminus B_R(y)}\rho,\]

Here we note that $\vert D\rho\vert\leq\frac{\vert x-y\vert}{2(s-t)}$. Then inserting $\psi\rho$ to (\ref{E:Rewrite weak form of Brakke Flow}) gives us
\begin{equation*}
\begin{split}
&\int \psi\rho d\mu_{t_2}+\int_{t_1}^{t_2}\int\psi\rho\left\vert \vec H+ \frac{(x-y)^\bot}{2(s-t)}+\frac{D\psi}{\psi}-\frac{\beta}{2}\right\vert^2 d\mu_t dt\\
&\leq\int\psi\rho d\mu_{t_1}+\int_{t_1}^{t_2}\int_{B_{2R}(y)\setminus B_R(y)} C(\frac{1}{R^2}+\frac{1}{s-t_2})\rho d\mu_t dt+\int_{t_1}^{t_2}\int\phi\frac{\vert \beta\vert^2}{4}\rho d\mu_{t}dt
\end{split}
\end{equation*}
Note we have already proved that $\int\rho d\mu_{t}\leq\infty$ for all $t\in[t_1,t_2]$ in Corollary \ref{Cor: Finite gaussian}, then by monotone convergence theorem and the dominated convergence theorem we get the desired result.
\end{proof}

\begin{remark}
By triangle inequality, we may rewrite the monotonicity formula into the following way:
\begin{corollary}
For any Brakke flow $\{\mu_t\}$ satisfies (\ref{E:Weak form of Brakke Flow}), any $(y,s)\in\mb R^N\times [T_0,\infty)$ and all $T_0\leq t_1\leq t_2<\min\{s,T_1\}$, we have
\begin{equation}\label{E:Monotonicity formula for addtional forces at one side}
\begin{split}
&\int\rho(x,t_2)d\mu_{t_2}+\int_{t_1}^{t_2}\int\rho(x,t)\left\vert \vec H+\frac{(x-y)^{\bot}}{2(s-t)}\right\vert^2d\mu_t dt\\
&\leq\int\rho(x,t_1) d\mu_{t_1}+\int_{t_1}^{t_2}\int\frac{\vert \beta\vert^2}{2}\rho(x,t) d\mu_{t}dt.
\end{split} 
\end{equation}
\end{corollary}
This form of monotonicity formula sometimes is even more useful, because it is just the classical Huisken's monotonicity formula with only one extra term.
\end{remark}

Together with Lemma \ref{L:Comparison Lemma for Integration}, note $\beta$ is bounded, we have the following growth bound of $\int\rho d\mu_t$:

\begin{corollary}\label{Cor: Growth of weighted integral and entropy}
For $T_0\leq t_1\leq t_2\leq T_1$, we have
\begin{equation}
\int\rho(x,t_2)d\mu_{t_2}\leq e^{\frac{\Vert\beta\Vert_{L^\infty}^2}{4}(t_2-t_1)}\int\rho(x,t_1)d\mu_{t_1}.
\end{equation}
Moreover taking the supremum of the above inequality for all $(y,s)\in\mb R^N\times(t,\infty)$ gives the entropy bound
\begin{equation}
\lambda(\mu_t)\leq e^{\frac{\Vert\beta\Vert_{L^\infty}^2}{4}(t-T_0)} \lambda(\mu_{T_0}).
\end{equation}
\end{corollary}

\subsection{Parabolic Rescaling}\label{SS:Parabolic Rescaling}
In this section we introduce the parabolic rescaling and discuss some property of it. For given Brakke flow with additional forces $\mu_t$, given $(y,s)\in\mb R^N\times\mb R$ and $\alpha>0$, define
\begin{equation}
\mu_t^{(y,s),\alpha}(A)=\alpha^{-m}\mu_{s+\alpha^2 t}(y+\alpha A),~ t\in[-\frac{s}{\alpha^2}+T_0,-\frac{s}{\alpha^2}+T_1).
\end{equation}
Here $A$ is any measurable subset of $\mb R^N$. If we fix $(y,s)$ then we will simply use $\mu_t^\alpha$ to denote $\mu_t^{(y,s),\alpha}$. For $\{M_t\}_{t\in[T_0,T_1)}$ be a family of submanifold satisfies mean curvature flow equation, the rescaling is defined to be
\[M_t^{(y,s),\alpha}=\alpha^{-1}(M_{s+\alpha^2 t}-y).\]
If we fix $(y,s)$ then we will simply use $M_t^\alpha$ to denote $M_t^{(y,s),\alpha}$. The rescaling of the Brakke flow with additional forces is again a Brakke flow with additional forces, but the force term $\beta$ changes. After the paraolic scaling, the additional forces $\beta^\alpha$ for $\mu_t^\alpha$ becomes 
\[\beta^\alpha(x,\cdot)=\alpha\beta(y+\alpha x,\cdot).\]

Then the monotonicity formula becomes
\begin{equation}\label{E:Monotonicity formula for blow up sequence}
\begin{split}
&\int\rho(x,t_2)d\mu^\alpha_{t_2}+\int_{t_1}^{t_2}\int\rho(x,t)\left\vert \vec H+\frac{(x-y)^{\bot}}{2(s-t)}-\alpha\frac{\beta}{2}\right\vert^2d\mu^\alpha_t dt\\
&\leq\int\rho(x,t_1) d\mu^\alpha_{t_1}+\int_{t_1}^{t_2}\int\alpha^2\frac{\vert \beta\vert^2}{4}\rho(x,t) d\mu^\alpha_{t}dt,
\end{split} 
\end{equation}
We can rewrite it into the following form from (\ref{E:Monotonicity formula for addtional forces at one side}):
\begin{equation}\label{E:Monotonicity formula for blow up sequence 2}
\begin{split}
&\int\rho(x,t_2)d\mu^\alpha_{t_2}+\int_{t_1}^{t_2}\int\rho(x,t)\left\vert \vec H+\frac{(x-y)^{\bot}}{2(s-t)}\right\vert^2d\mu^\alpha_t dt\\
&\leq\int\rho(x,t_1) d\mu^\alpha_{t_1}+\int_{t_1}^{t_2}\int\alpha^2\frac{\vert \beta\vert^2}{2}\rho(x,t) d\mu^\alpha_{t}dt,
\end{split} 
\end{equation}

Now let us study some quantities after parabolic rescaling.

\subsubsection{Entropy Bound After Parabolic Rescaling}
By Corollary \ref{Cor: Growth of weighted integral and entropy} we have that for $T_0-s/\alpha^2\leq t_1\leq t_2\leq T_1-s/\alpha^2$:
\begin{equation}\label{E:Entropy bound after rescaling}
\int\rho(x,t_2)d\mu^\alpha_{t_2}\leq e^{\frac{\alpha^2\Vert\beta\Vert_{L^\infty}^2}{4}(t_2-t_1)}\int\rho(x,t_1)d\mu^\alpha_{t_1}.
\end{equation}
Since $\int\rho(x,-\frac{s}{\alpha^2}+T_0)d\mu^\alpha_{-\frac{s}{\alpha^2}+T_0}$ is just $\int \rho(x,T_0)d\mu_{T_0}$ with another weight $\rho$, so it is bounded by the uniform entropy $\lambda(\mu_{T_0})$. Then we get a uniform bound
\begin{equation}
\int\rho(x,t)d\mu^\alpha_t \leq C\lambda(\mu_{T_0})
\end{equation} 
in particular for $t\in[T_0',T_1']$ for fixed $T_0',T_1'$, where $C$ is independent of $\alpha$. This implies the entropy bound of $\mu^\alpha_t$ for $t\in [T_0',T_1']$.
\begin{equation}\label{E: Entropy growth bound}
\lambda(\mu^\alpha_t)\leq C\lambda(\mu_{T_0}).
\end{equation}
By Theorem \ref{T:area and entropy are equivalent}, we know that $\mu^\alpha_t$ also has area growth bound for $t\in [T_0',T_1']$:
\begin{equation}\label{E: Area growth uniform bound after rescaling}
\sup_{x\in\mb R^N}\sup_{R>0}\frac{\mu(B_R(x))}{R^m}\leq C.
\end{equation}
This bound plays an important role in the following section.

\subsubsection{Integral Bound of $|H|^2$ After Parabolic Rescaling}
We may also use (\ref{E:Monotonicity formula for blow up sequence 2}) to estimate $H^2$. First let us rewrite (\ref{E:Monotonicity formula for blow up sequence 2}) to the following form
\begin{equation}
\begin{split}
&\int_{t_1}^{t_2}\int\rho(x,t)\left\vert \vec H+\frac{(x-y)^{\bot}}{2(s-t)}\right\vert^2d\mu^\alpha_t dt\\
&\leq\int\rho(x,t_1) d\mu^\alpha_{t_1}-\int\rho(x,t_2)d\mu^\alpha_{t_2}+\int_{t_1}^{t_2}\int\alpha^2\frac{\vert \beta\vert^2}{2}\rho(x,t) d\mu^\alpha_{t}dt,
\end{split} 
\end{equation}
Let $t_1=-2$, $t_2=-1$ and $(y,s)=(0,0)$, together with (\ref{E:Entropy bound after rescaling}), we get 
\begin{equation}\label{E:H^2+x/2t is bounded, turns to 0}
\int_{-2}^{-1}\int\rho_{0,0}(x,t)\left\vert \vec H+\frac{x^{\bot}}{-2t}\right\vert^2d\mu^\alpha_t dt=\delta(\alpha)\to 0\text{ as $\alpha\to0$.}
\end{equation}
Thus this estimate together with the naive triangle inequality gives the following estimate
\[
\int_{-1-\tau}^{-1}\int\vert \vec H\vert^2\chi_{B_r(x)}d\mu^\alpha_t dt\leq C\tau r^m R^2+ C\delta(\alpha),
\]
where $B_r(x)$ is a small ball which lies inside $B_R$.

Now we follow Ilmanen \cite[Lemma 6]{ilmanen1995singularities} to improve this estimate slightly.

\begin{lemma}\label{L:Improved H^2 bound after rescaling}
For any $B_r(x)\subset B_R$, we have
\begin{equation}\label{E:H^2 integral bound from -1-sigma to -1}
\int_{-1-\tau}^{-1}\int|\vec H|^2 d\mu^{\alpha}_{t}\leq C \tau(r^m+r^{m-1}R)+\delta(\alpha).
\end{equation}
\end{lemma}

\section{Weak Blow Up Limit}\label{S:Weak Blow Up Limit}
The main goal of this section is to prove that the blow up limit of Brakke flow with additional forces exists, and we will discuss some property of the limit.

\subsection{Allard's Compactness Theorem}
In this section we first introduce some notations in geometric measure theory, and state the Allard's compactness theorem. Let $V$ be a rectifiable $k$-varifold. Let $\Phi:\mb R^N\to \mb R^N$ be a $C^1$ map, then the \emph{push-forward} $\Phi_\sharp(\mu)$ is defined to be the varifold such that
\[
\Phi_\sharp(V)(f)=\int f(x,S)|J_S\Phi(x)|d V(x,S)
\]
for all test functions $f\in C_c^0(G^{(k,N)},\mb R)$. Here $|J_S\Phi(x)|$ is the Jacobian defined by
\[
|J_S(\Phi(x))|=\sqrt{\det((d\Phi(x)|_{S})^T\circ d\Phi(x)|_S)}, x\in\mb R^N, S\in G_x^{(k,N)}.
\]
we define the \emph{first variation} $\delta V$ of $V$ is the linear functional on the sections of $C_c^1(T\mb R^N)$ given by
\[
\delta V(Y)=\left.\frac{d}{dt}\right\vert_{t=0}\Phi_\sharp^t(V)(U),
\]
where $Y\in C_c^1(T\mb R^N)$ such that $\text{spt}Y\subset\subset U\subset\subset\mb R^N$, and $\{\Phi^t\}_{t\in(-\eps,\eps)}$ is a family of diffeomorphisms supported in $U$ with $\frac{\partial}{\partial t}|_{t=0}\Phi^t=Y,\Phi^0=id$.

There is a one to one correspondence between integer $k$-rectifiable Radon measures $\mu$ on $\mb R^N$ to the integer rectifiable $k$-varifold $V$ on $G^{(k,N)}$. From now on we will use $V_\mu$ to denote the $k$-varifold associated to $\mu$. 

In particular, if $\mu$ is a $k$-rectifiable Radon measure, and assume $V_\mu$ is the associate varifold, then the first variation formula becomes
\[
\delta V_\mu(Y)=\int-\vec H\cdot Y d\mu +\delta \mu_{sing}(Y),
\]
where $\mu_{sing}$ is the singular set of $V_\mu$.

Now we can state the Allard's compactness theorem.

\begin{theorem}\label{T:Allard's compactness theorem}
Let $\{\mu_i\}$ be a family of $k$-rectifiable Radon measures with
\[
\sup(\mu_i(U)+|\delta V_{\mu_i}|(U))<\infty \text{ for each $U\subset\subset\mb R^N$}.
\]
Then there exists an integer $k$-rectifiable Radon measure $\mu$ and a subsequence of $\{\mu_i\}$ (denoted by $\mu_j$) such that
\begin{enumerate}
\item $\mu_j\to\mu$ as Radon measures on $\mb R^N$,
\item $V_{\mu_j}\to V_\mu$ as Radon measures on $G^{(k,N)}$,
\item $\delta V_{\mu_j}\to \delta _{\mu}$ as $T\mb R^N$ valued Radon measures,
\item $|\delta V_\mu|\leq\liminf_{j}|\delta V_{\mu_j}|$ as Radon measures.
\end{enumerate}
\end{theorem}

For a proof, see \cite{simon1983lecture}.

\subsection{Compactness of Brakke Flow with Additional Forces}
In this section we will recall and sketch the proof of compactness for integral Brakke flows by Ilmanen in \cite[Section 7]{ilmanen1994elliptic}, and discuss how to generalize it to the Brakke flows with additional forces:

\begin{theorem}\label{T:Compactness of Brakke flow with additional forces}
Let $\{\mu_t^i\}_{t\in[T_0,T_1]}$ be a sequence of Brakke flows with additional forces $\beta^i$ in $\mb R^N$, and assume $\Vert\beta^i\Vert_{L^\infty}\to 0$ as $i\to \infty$. If 
\[\sup_{i,t}\mu_t^i(U)\leq C(U)<\infty\]
for arbitrary $U\subset\subset\mb R^N$, then there is a subsequence $\{\mu^{i_j}\}$ and a Brakke flow $\mu_t$ such that $\mu_t^{i_j}\to \mu_t$ as Radon measure for each $t\in[T_0,T_1]$.
\end{theorem}

\begin{proof}[Sketch of the proof]
We first sketch the proof for Brakke flows, and then we point out some necessary modification in the proof to show that the proof works for Brakke flow with additional forces whose $L^\infty$ norm goes to $0$. Follow \cite[Section 6.2]{ilmanen1994elliptic} we define 
\[\mc B(\mu,\phi)= \int-\phi H^2+\nabla^\bot \cdot \vec H d\mu, \phi\in C_c^2(\mb R^N,[0,\infty)).\]
We will use $D^+$ and $D^-$ to denote the left and right upper derivatives respectively, and $\bar D_t\mu_{t_0}(\phi)$ to be the maximum of $D^+_{t_0}\mu_t(\phi)$ and $D^-_{t_0}\mu_t(\phi)$.

\emph{Step 1:} The mass bound $\sup_{i,t}\mu_t^i(U)\leq C(U)<\infty$ implies we can choose a subsequence of $\mu_t^i$ converge to $\mu_t$ on a countably dense subset $B_1$ of $[T_0,T_1]$. Moreover by \cite[7.2]{ilmanen1994elliptic}, $\mu_t^i-C_3(\phi)t$ is nonincreasing implies the limit $\mu_t-C_3(\phi)t$ is nonincreasing on $B_1$. Here $C_3(\phi)$ is a constant depending on $\phi$. 
\vskip 8pt
\emph{Step 2:} For any $t\in [T_0,T_1]\backslash B_1$, define $\mu_t$ be the limit of the convergent subsequence of previous subsequence $\{\mu_t^i\}$. Then $\mu_t$ is defined for all $t\in[T_0,T_1]$. Note $\mu_t^i-C_3(\phi)t$ is still nonincreasing for $B_1\cup\{t\}$, thus $\mu_t^i-C_3(\phi)t$ is nonincreasing for any $\phi\in C_c^2(\mb R^N,[0,\infty))$ and all $t\in[T_0,T_1]$.

Let $\Psi$ be a countable dense subset of $C_c^2(\mb R^N,[0,\infty))$. Then by \cite[7.2]{ilmanen1994elliptic}, for each $\psi\in\Psi$ there is a co-countable set $B_\psi\subset[T_0,T_1]$ such that $\mu_t(\psi)$ is continuous at each $t\in B_\psi$. Define $B_2=\cap B_\psi$ which is a co-countable set of $[T_0,T_1]$, then we know that $\mu_t$ is continuous at all $t\in B_2$. Thus we know that $\mu_t$ is uniquely determined by the value on the countable dense set $B_1$, independent the choice of subsequence for that $t$. Hence the full sequence converges:
\[\mu_t^i\to \mu_t, t\in B_2.\]

Then $B_2$ is co-countable implies we can choose a further subsequences and by diagonalize argument to obtain a subsequance (still denoted by $\{\mu_t^i\}$) to re-define $\mu_t=\lim_{i\to \infty}\mu_t^i$ for $t\in[T_0,T_1]\backslash B_2$. Thus we construct $\mu_t$ defined for all $t\in[T_0,T_1]$ and $\mu_t-C_3(\phi)t$ is nonincreasing for any $\phi\in C_c^2(\mb R^N,[0,\infty))$.
\vskip 8pt
\emph{Step 3:} Finally we check $\mu_t$ is a Brakke flow, i.e. for fixed $\phi\in C_c^2(\mb R^N,[0,\infty))$ and $t_0\in[T_0,T_1]$, $\bar D_{t_0}\mu_t(\phi)\leq\mc B(\mu_{t_0},\phi)$. Let us prove that $D^+_{t_0}\mu_t(\phi)\leq \mc B(\mu_t,\phi)$ and the other case for $D^-$ is similar. If $D^+_{t_0}\mu_t(\phi)=-\infty$ then we are done; if $D^+_{t_0}\mu_t(\phi)>-\infty$, then there exists $t_q\downarrow t_0$ and $\eps_q\downarrow t_0$ such that
\[D^+_{t_0}\mu_t(\phi)-\eps_q\leq\frac{\mu_{t_q}(\phi)-\mu_{t_0}(\phi)}{t_q-t_0}.\]
Since $\mu^i_t\to mu_t$, then we may assume there is a sequence $r_q\to \infty$ such that
\[D^+_{t_0}\mu_t(\phi)-2\eps_q\leq \frac{\mu_{t_q}^{r_q}-\mu_{t_0}^{r_0}}{t_q-t_0}\leq \frac{1}{t_q-t_0}\int_{t_0}^{t_q}\bar D_t\mu_t^{r_q}dt.\]
So we may choose $s_q\in[t_0,t_q]$ with 
\[D^+_{t_0}\mu_t(\phi)-2\eps_q\leq \bar D_t\mu_{s_q}^{r_q}(\phi)\leq\mc B(\mu_{s_q}^{r_q},\phi).\]
There is a subsequence of $\mu_{s_q}^{r_q}$ converges to a Radon measure $\mu$. Then we have
\[D^+_{t_0}\mu_t(\phi)\leq\mc B(\mu,\phi).\]
\vskip 8pt
\emph{Step 4:} Now the only thing remains to check is 
\[\mu=\mu_{t_0}\lfloor\{\phi>0\}.\]
Let $\psi\in C_c^2(\{\phi>0\},[0,]infty))$ and we may assume $\psi\leq \phi$ by scaling. Then we can check $D^+_{t_0}\mu_t(\psi)>-\infty$. $\mu_t(\psi)-C_3(\psi)t$ is nonincreasing implies that $\mu_{t_0}(\psi)\leq\lim_{t\downarrow t_0}\mu_t(\psi)$.

Then we show $\mu_{t_0}(\psi)=\mu(\psi)$. For fixed $t>t_0$, for $q$ large enough we have $t_0<s_q<t$, then $\mu_t^i-C_3(\psi)t$ is nonincreasing implies
\[\mu_t^{r_q}(\psi)+C_3(\psi)(s_q-t)\leq \mu_{s_q}^{r_q}(\psi)\leq\mu_{t_0}^{r_q}(\psi)+C_3(\psi)(s_q-t_0).\]
Let $q\to\infty$ and $s_q\downarrow t_0$ we obtain
\[\mu_{t_0}(\psi)\leq \mu(t)\leq \mu_{t_0}(\psi).\]
Thus $\mu_{t_0}(\psi)=\mu(\psi)$. Then we conclude that $\mu=\mu_{t_0}\lfloor\{\phi>0\}$s. 
\vskip 8pt
\emph{Modification to Brakke flows with additional forces:} Essentially the proof is the same. We need to check the lemmas in \cite[Section 6,7]{ilmanen1994elliptic} holds for Brakke flows with additional forces whose $L^\infty$ norm turns to $0$, and additional forces would only add a small term to the constants in the theorem. Also when taking limits in many places, $\beta\to 0$ will eliminate all the terms which are generated by the additional forces. So this proof is still valid for Brakke flows with additional forces whose $L^\infty$ norm turns to $0$.
\end{proof}

\subsection{Existence of Weak Blow Up Limit}
A direct corollary of the compactness Theorem \ref{T:Compactness of Brakke flow with additional forces} is the weak blow up limit always exists. More precisely, suppose $\mu_t$ is a Brakke flow with additional force, then for any fixed $(y,s)\in\mb R^N\times\mb R$ and sequence $\alpha_i\to 0$ where exists a subsequence still denoted by $\alpha_i$ such that
\[\mu_t^{(y,s),\alpha_i}\to \nu_t\]
as Radon measures for all $t\in\mb R^N$. Here we need that fact in Section \ref{SS:Parabolic Rescaling} that $\mu_t^{(y,s),\alpha_i}$ has bounded entropy (hence bounded area growth) on every bounded interval $[T_0',T_1']\subset\mb R^N$. Moreover, $\nu_t$ is a Brakke flow with finite entropy. $\nu_t$ is called the \emph{tangent flow} of $\mu_t$ at $(y,s)$.

Now we prove the following fact, cf. \cite[Lemma 8]{ilmanen1995singularities}
\begin{theorem}
Given $t<0$, then 
\begin{equation}
\nu_t(A)=\lambda^{-k}\nu_{\lambda^2t}(\lambda A)
\end{equation}
holds for all $\lambda>0$ and $A\subset\mb R^N$ is measurable. Moreover $\nu_{-1}$ satisfies
\[\vec H(x)+\frac{x^\bot}{2}=0,\text{ $\nu_{-1}$-almost everywhere $x$.}\]
\end{theorem}

\begin{remark}
This theorem implies that if we want to show the blow up limit is smooth, we only need to prove the smoothness for a single time slice.
\end{remark}

\begin{proof}
We first prove that the $F$-functional converges
\begin{equation}
\int \rho_{0,0}d\nu_t(x)=\lim_{i\to\infty}\int\rho_{0,0}d\mu_{t}^{\alpha_i}.
\end{equation}

Note the convergence of $\mu_t^{\alpha_i}$ as Radon measures only implies the integral of a compactly supported continuous function converges. We need to use the exponential decay property of $\rho_{0,0}$. By (\ref{E: Area growth uniform bound after rescaling}) we know that $\mu_t^{\alpha_i}$ have uniform area growth bound. So we have
\begin{equation}
\begin{split}
\int_{\mb R^N\backslash B_R}\rho_{0,0}d\mu_t^{\alpha_i}&\leq \frac{C}{(-t)^{m/2}}\sum_{j=1}^\infty \int_{B_{R^{j+1}\backslash B_{R^j}}} e^{-R^{2j}/(-4t)}d\mu_t^{\alpha_j}\\
&\leq \frac{C}{(-t)^{m/2}}\sum_{j=1}^\infty R^{(j+1)m}e^{-R^{2j}/(-4t)}\\
&=C f(R,t),
\end{split}
\end{equation}
where $f(R,t)$ is a function depending on $R,t$, and converge to $0$ as $R\to \infty$ for fixed $t$. Also note that the above computation is also valid for $\nu_t$. Thus we may claim that 

\[\int \rho_{0,0}d\nu_t(x)=\lim_{i\to\infty}\int\rho_{0,0}d\mu_{t}^{\alpha_i}.\]

This identity together with (\ref{E:Entropy bound after rescaling}) immediately implies that $\int \rho_{0,0}d\nu_t$ are he same for all $t<0$. Thus the monotonicity formula for Brakke flow, i.e. Theorem \ref{T:monotonicity formula} with $\beta=0$ implies that $\nu_t$ satisfies the self-shrinkers equation
\begin{equation}\label{E: Self-shrinker equation for all t}
\vec H+\frac{x^\bot}{-2t}=0,\text{ $\nu_t$-almost everywhere $x$}.
\end{equation}
Moreover this implies that $\nu_t$ is rectifiable.

Finally we show the self similarity of $\nu_t$. Define $\tilde\nu_t(A)=(-t)^{-k/2}\nu_{t}((-t)^{1/2} A)$ for $t<0$, and we only need to show $\tilde\nu_t$ is constant to conclude the self similarity. Suppose $t_1\leq t_2<0$, then by (\ref{E:Weak form of Brakke Flow}) with $\beta=0$, for any $\phi\in C_c^2(\mb R^N,[0,\infty))$ we have
\begin{equation}
\begin{split}
&\int\phi d\tilde\nu_{t_2}-\int\phi d\tilde\nu_{t_1}=\int \frac{\phi((-t_2)^{1/2}x)}{(-t_2)^{m/2}}d\nu_{t_2}-\int \frac{\phi((-t_2)^{1/2}x)}{(-t_2)^{m/2}}d\nu_{t_1}\\
&\leq \int_{t_1}^{t_2}\int \left(- \frac{\phi((-t)^{1/2}x)}{(-t)^{m/2}} |H|^2  +\vec H\cdot \frac{D\phi((-t)^{1/2}x)}{(-t)^{(m+1)/2}} + \frac{m}{2(-t)^{m/2+1}}\phi -\frac{1}{2(-t)^{1/2}}D\phi\cdot x \right) d\nu_t dt.
\end{split}
\end{equation}
Substituting from (\ref{E: Self-shrinker equation for all t}) and the first variational formula we get $\int\phi d\tilde\nu_{t}$ is nonincreasing in $t$. Assume without lost of generality $\phi<\rho_{0,0}$ and apply the same calculation to $\rho_{0,0}-\phi$, then $\int(\rho_{0,0}-\phi) d\tilde\nu_{t}$ is nonincreasing. Here we use the fact that $\rho_{0,0}$ has exponential decay to validate the insertion as a test function. Then we conclude that $\tilde\nu_t$ is a constant. This completes the proof.

\end{proof}
\section{Tools for Analysis of Surfaces}\label{S:Tools for Analysis of Surfaces}
From now on, we will concentrate on the case of mean curvature flow of surfaces in $\mb R^3$ with additional forces. In this section we introduce some tools which are important to the study of surfaces in $\mb R^3$. Due to the limitation of the length of this paper, we will only state the theorems and discuss their applications in this paper.
\subsection{Local Gauss-Bonnet Theorem}
\begin{theorem}\label{T:Local Gauss-Bonnet}
Let $R>1$ and let $M$ be a surface properly immersed in $B_R$, then for any $\eps>0$,
\begin{equation}
(1-\eps)\int_{M\cap B_1}|A|^2d\mu \leq \int_{M\cap B_R}|\vec H|^2d\mu +8\pi g(M\cap B_R)-8\pi c'(M\cap B_R)+\frac{24\pi R^2}{\eps(R-1)^2}.
\end{equation}
Here $g(M\cap B_R)$ is the genus of $M\cap B_R$, $c'(M\cap B_R)$ is the number of components of $M\cap B_R$ that meet $B_1$, and $D'=\sup_{r\in[1,R]}\mc H^2(M\cap B_r)/(\pi r^2)$.
\end{theorem}
Note by Theorem \ref{T:area and entropy are equivalent} we know that we may replace $D'$ by the entropy of $M$. 

See \cite[Theorem 3]{ilmanen1995singularities} for the proof. Together with (\ref{E: Entropy growth bound}) and (\ref{E:H^2 integral bound from -1-sigma to -1}), integrating this estimate from $-1-\sigma$ to $-1$ gives the following estimate
\begin{theorem}\label{T:Result of Local Gauss-Bonnet: A bound}
Let $M_t$ be a mean curvature flow of $2$-manifolds smoothly and properly immersed in $\mb R^3$, then for every $B_R(x)\subset\mb R^3$ with $R<1/2$,
\begin{equation}
\int_{-1-\tau}^{-1}\int_{M^\alpha_t\cap B_R(x)}|A|^2\leq C\tau(R^2+8\pi g(M_0)+\lambda(M_0))+\delta(\alpha).
\end{equation}
\end{theorem}

So the local Gauss-Bonnet Theorem gives us an uniformly bound for the scale invariant curvature integral.

\subsection{Allard's Regularity Theorem}
Allard's regularity theorem is significant in the study of geometric measure theory. See \cite{allard1972first}, \cite{simon1983lecture}.

\begin{theorem}[Allard's Regularity Theorem]\label{T: Allard's regularity theorem}
There exists $\eps_A>0$ and $\sigma_A>0$ both depending on $n,k$ with the following property. Suppose $\mu$ is an integer $k$-rectifiable Radon measure such that $|\vec H|\in L^1_{\text{loc}}(\mu)$. If there is $\eps<\eps_A$ and $0\in \text{spt}\mu$, $r>0$ such that $|\vec H|\leq \eps/r$, for $\mu$-a.e. $x\in B_r$, and $\mu(B_r)\leq (1+\eps)\omega_k r^k$, then there is a $k$-plane $T$ through $0$, and domain $\Omega\subset T$, and a $C^{1,\alpha}$ vector valued function $u:\Omega\to T^\bot$ such that
\[
\mbox{spt}\mu\cap B_{\sigma r}=\text{graph}(u)\cap B_{\sigma r}.
\]
Here $\text{graph}(u)=\{x+u(x):x\in\Omega\}$ and
\[
\sup\left\vert\frac{u}{r}\right\vert+\sup|Du|+r^\alpha[Du]_{\alpha}\leq C\eps^{1/4n},
\]
where $[Du]_\alpha=\sup\frac{|Du(x)-Du(y)|}{|x-y|^\alpha}$.
\end{theorem}

Later we will use Allard's regularity theorem to prove that if a surface has area closed to a disk with the same radius, then the surface is actually regular ($C^{1,\alpha}$). In particular, we will prove our blow up limit satisfies such condition besides discrete points set, hence it must be $C^{1,\alpha}$ surface locally. Then standard elliptic theory may promote the regularity to $C^\infty$. 

\subsection{Simon's Graph Decomposition Theorem}
\cite{Simon-Willmore} proved the following theorem:

\begin{theorem}[Simon's graph decomposition theorem]\label{T:Simon's Graph Decomposition Theorem}
For $D>0$, there is $\eps_S=\eps_S(n,D)$ such that if $M$ is a smooth $2$-manifold properly embedded in $B_R$ and
\[
\int_{M\cap B_R}|A|^2\leq\eps^2\leq\eps_S^2,\mc H^2(M\cap B_R)\leq D\pi R^2,
\]
then there are pairwise disjoint closed disks $\overline{P}_1,\overline{P}_2,\cdots,\overline{P}_N$ in $M\cap B_R$ such that
\begin{equation}
\sum_m\text{diam}(P_m)\leq C(n,D)\eps^{1/2}R
\end{equation}
and for any $S\in[R/4,R/2]$ such that $M$ is transverse to $B_S$ and $\partial B_S\cap\cup_m P_m=\emptyset$, we have
\[
M\cap B_S=\cup_{l=1}^mD_l
\]
where each $D_i$ is an embedded disk. Furthermore, for each $D_l$ there is a $2$-plane $L_l\subset\mb R^3$, a simply connected domain $\Omega_l\subset L_l$, disjoint closed balls $\bar B_{l,p}\subset\Omega_l$, $p=1,\cdots,p_l$ and a function
\[
u_l:\Omega\backslash \cup \bar B_{l,p}\to L_l^\bot
\]
such that 
\[
\sup\left\vert\frac{u_l}{R}\right\vert+|Du_l|\leq C(n,D)\eps^{1/6},
\]
and 
\[
D_l\backslash\cup_m \overline{P}_m=\text{graph}(u_l|_{\Omega_l\backslash \cup_p\bar B_{l,p}}).
\]
\end{theorem} 

\begin{corollary}\label{Cor: corollary of Simon's graph decomposition}
Under the above hypothesis with the assumption that $M$ has the area growth bound 
\[
\sup_{x\in\mb R^N}\sup_{s>0}\frac{\mu(B_s(x))}{s^m}\leq C.
\]

Suppose $x\in M\cap B_R$ such that $B_{2r}(x)\subset B_R(x)$, then the connected component $M'$ of $M\cap B_r(x)$ containing $x$ is embedded and satisfies
\[
\pi r^2(1-C(n,D)\eps^{2r})\leq \mc H^2(M')\leq \pi r^2(1+C(n,D)\eps^\gamma).
\]

\end{corollary}
\begin{proof}
By the above theorem, $M\cap B_r(x)=\cup_{l=1}^m D_l$, so $M'$ is one of $D_l$'s, let's say $D_k$. Then we have
\[\mc H^2(M')=\mc H^2(D_k)\geq \mc H^2(\text{graph}(u_l|_{\Omega_l\backslash \cup_p\bar B_{l,p}}))\geq \pi r^2(1-C(n,D)\eps^{2\gamma})\]
as the lower bound, and
\[
\mc H^2(M')\leq \mc H^2(\text{graph}(u_l|_{\Omega_l\backslash \cup_p\bar B_{l,p}}))+\mc H^2(\cup_m P_m)\leq \pi r^2(1+C(n,D)\eps^\gamma),
\]
by using the area growth bound condition to bound the area of $\mc H^2(\cup_m P_m)$. Here we use the algebraic inequality $\sum_m a_m^2\leq (\sum_m a_m)^2$.
\end{proof}

Simon's graph decomposition Theorem provides a way to decompose the surface in a small ball into several sheets. Morevoer, together with Allard's regularity theorem they prove that each sheet has nice regularity. Then we are able to analyze the convergence behavior on each single sheet. 

\section{Smoothness of Blow Up Limit}\label{S:Smoothness of Blow Up Limit}
In this section we will prove the weak blow up limit $\nu_{t}$ we get in Section \ref{S:Weak Blow Up Limit} is actually a smooth mean curvature flow. Since in Section \ref{S:Weak Blow Up Limit} we have already proved that $\mu_t$ is a homothetic solution to mean curvature flow, so we only need to prove the following theorem:

\begin{theorem}\label{T:Blow up limit is smooth}
Suppose $\{M_t\}_{t\in[0,s)}$ is smooth mean curvature flow of surfaces in $\R^3$ with additional forces, and $M_0$ is a closed smoothly embedded surface with finite entropy and finite genus. Let $\nu_t$ be the blow up limit of $\{M_t\}$ at the first singular time, then the support of $\nu_{-1}$ is a smoothly embedded self-shrinker in $\mb R^3$.
\end{theorem}

\begin{proof}[Proof of Theorem \ref{T:Blow up limit is smooth}]
The proof follows \cite{ilmanen1995singularities}. From now on a sequence appears in later lines may be a subsequence of a sequence in front lines, even they use the same indexes. Let $N$ be the support of $\nu_{-1}$.

\vspace{4pt}
\emph{Step 1: Select well-controlled time slices.} Let us consider the blow up sequence $M^{\alpha_j}_{t_j}$ with $\alpha_j\to\infty$ large enough such that $[-2,0]$ is a well-defined time range of $M^{\alpha_j}_{t_j}$. Then for any given $x\in\mb R^3$ and $R\leq 1/2$ by Theorem \ref{T:Result of Local Gauss-Bonnet: A bound} we get
\begin{equation}
\int_{-1-\tau_j}^{-1}\int_{M^{\alpha_j}_{t}\cap B_R(x)}|A|^2\leq C\tau(R^2+8\pi g(M_0)+\lambda(M_0))+\delta(\alpha_j).
\end{equation}

By (\ref{E:H^2+x/2t is bounded, turns to 0}) we get
\begin{equation}
\int_{-1-\tau_j}^{-1}\int_{M^{\alpha_j}_{t}}\rho_{0,0}(x,t)\left\vert \vec H+\frac{x^{\bot}}{-2t}\right\vert^2\leq\delta(\alpha_j).
\end{equation}
Then note that $\int_{-1-\tau_j}^{-1}f\leq a$ implies at least on a set of measure $2\tau_j/3$, $f\leq 3a/2$, we can choose $t_j\in[-1,\tau_j]$ such that
\begin{equation}\label{E: A is bounded in a rescaled time slice}
\int_{M^{\alpha_j}_{t_j}\cap B_R(x)}|A|^2\leq C(R^2+8\pi g(M_0)+\lambda(M_0))+\frac{\delta(\alpha_j)}{\tau_j},
\end{equation}

\begin{equation}\label{E: H^2+x/2t is bounded, turns to 0 in the proof}
\int_{M^{\alpha_j}_{t_j}}\rho_{0,0}(x,t)\left\vert \vec H+\frac{x^{\bot}}{-2t}\right\vert^2\leq\frac{\delta(\alpha_j)}{\tau_j}.
\end{equation}

We select $\tau_j\to 0$ and $\alpha_j\to\infty$ such that $\delta(\alpha_j)/\tau_j\to 0$. Let us define $\mu_j$ to be the radon measure induced by $M_{t_j}^{\alpha_j}$, and $M_j$ be the surface $M_{t_j}^{\alpha_j}$. By the existence of weakly blow up limit, we know that $\mu_j$ converge to $\nu_{-1}$ as Radon measures.

\vspace{4pt}
\emph{Step 2: Smooth limit away from curvature concentration points.} Define $\sigma_j=|A|^2\mc H^2\lfloor M_{t_j}^{\alpha_j}$. Then from the above choice of $\alpha_j$ and $t_j$, (\ref{E: A is bounded in a rescaled time slice}) implies that $\sigma_j(B_R)$ is uniformly bounded. Hence there is a subsequence of $\sigma_j$ converge to a Radon measure $\sigma$. We may run this argument to a family of balls $B_R$ covering $\mb R^3$ and select a subsequence by the diagonal criteria to get $\sigma$ defined on whole $\mb R^3$.

Now we fix $\eps_0>0$ such that $\eps_0\leq\min\{\eps_A,\eps_S,(\eps_S/C(n,D))^{1/\gamma}\}$, where $\eps_A$ comes from Theorem \ref{T: Allard's regularity theorem} and $\eps_S$ comes from Theorem \ref{T:Simon's Graph Decomposition Theorem}. We define set $Q$ consists of points $q$ such that $\sigma(q)\geq\eps_0$, and call $Q$ the set of \emph{concentration points}. (\ref{E: A is bounded in a rescaled time slice}) implies that in $B_R$, there are at most $C(R^2+8\pi g(M_0)+\lambda(M_0))/\eps_0$ number of points in $Q$. Hence $Q$ is discrete in $\mb R^3$.

We claim $N\cap B_R(x)$ is smoothly embedded away from $Q$. In order to prove this, let $p\in (N\cap B_R(x))\backslash Q$ be arbitrary point. Since $p$ is not a concentration point, there exists $r>0$ such that $B_r(p)\subset B_R(x)$, $r<\eps_0/R$ and $\sigma(B_r(p))<\eps_0^2$. The last condition implies that $\sigma_j(B_r(p))<\eps_0^2$ for sufficiently large $j$. Then Simon's decomposition Theorem \ref{T:Simon's Graph Decomposition Theorem} implies that there is $s_j\in[r/4,r/2]$ such that $M_k\cap B_{s_j}(p)$ is a union of distinct connected disks $D_{j,l}$. Moreover, Corollary \ref{Cor: corollary of Simon's graph decomposition} implies that each $D_{j,l}$ satisfies
\[\mc H^2(D_{j,l}\cap B_\rho(y))\leq \pi(1+C(n,D)\eps_0^\gamma)\rho^2,\]
for any $y\in M_j$ and $\rho+|y-p|\leq s_j$. We may choose a subsequence such that $s_k\to s\in[r/4,r/2]$ and $M_k\cap B_{s_j}(p)$ has equally number of disks, and $\mc H^2\lfloor D_{k.l}\to\nu_l$ as Radon measure for every $l$. By the lower semicontinuity of the quantity in (\ref{E: H^2+x/2t is bounded, turns to 0 in the proof}), $\nu_l$ weakly solves the self-shrinker equation
\[\vec H+\frac{x^\bot}{2}=0.\]

Then the hypothesis on $r$ implies that $|\vec H|\leq\eps/r$. Also the area bound of $D_{j,l}$ implies that
\[\nu_l(D_{j,l}\cap B_\rho(y))\leq \pi(1+C(n,D)\eps_0^\gamma)\rho^2,\]
for any $y$ in the support of $\nu_l$ and $\rho+|y-p|\leq s$. Then Allard regularity Theorem \ref{T: Allard's regularity theorem} implies that the support of $\nu_l\cap B_{s'}$ (here $s'$ is a multiple of $s$) is the graph of a $C^{1,\alpha}$ function defined over a domain in a $2$-plane. Then standard Schauder estimate implies that $u$ is actually a $C^\infty$ function because it satisfies the self-shrinker equation. Thus $\nu_l$ is the support of an immersed surface. Moreover, each $M_t$ is embedded implies that $\nu_l$ is the support of a smoothly embedded surface. $N\cap B_{s'}(p)=\cup_l \text{spt}\nu_l$ together with the maximum principle implies that $N\cap B_{s'}(p)$ is a smoothly embedded surface. 

This argument holds for any $p\in (N\cap B_R(x))\backslash Q$, and for any $x\in\mb R^3$. Thus $N\backslash Q$ is a smoothly embedded surface. 

\vspace{4pt}
\emph{Step 3: Smooth limit.} Finally we claim that $N$ is smoothly embedded across the concentration points. Note a self-shrinker in $\mb R^3$ is actually a minimal surface in $\mb R^3$ with Gaussian metric, i.e. $g_{ij}=e^{|x^2|/8}\delta_{ij}$, so $N\backslash Q$ is a minimal surface in a three manifold, then by \cite[Proposition 1]{choi1985space} we know that $N$ can be smoothly extended through $Q$. Moreover, by maximum principle, $N$ is still smoothly embedded after this extension. Therefore we finish the proof.
\end{proof}

\begin{remark}
The proof is still valid even we weaken the assumption that $\beta$ has bounded $L^\infty$ norm. Suppose $U\subset\R^3$ is an open subset and $\Sigma_t\subset U$ for all time $t$, then we only require $\beta$ has bounded $L^\infty$ norm on the $U$ to conclude the theorem.
\end{remark}

\section{Appendix}
\subsection*{Comparison Lemma}
In this section we prove the comparison lemma of the integral inequality which appears in section \ref{S:Mean Curvature Flow with Additional Forces}.

\begin{lemma}\label{L:Comparison Lemma for Integration}
Given $C\geq 0,D\geq 0,t_0\in[T_1,T_2]$. Suppose $f:[T_1,T_2]\to\mb R^+$ is a function such that the integral inequality
\[f(t)\leq f(t_0)+C\int_{t_0}^{t}f(s)ds\]
holds for all $t\in[t_0,T_2]$, then
\[f(t)\leq e^{C(t-t_0)}f(t_0).\]
\end{lemma}

\begin{proof}
Let 
\[g_\eps(t)=f(t)-(1+\eps)e^{C(t-t_0)}f(t_0).\]
Then we only need to show $g_\eps(t)< 0$ for any $\eps>0$, and then let $\eps\to 0$. Integrating $Cg$ from $t_0$ to $t$ gives
\begin{equation}
\begin{split}
C\int_{t_0}^tg_\eps(s)ds&=C\int_{t_0}^t f(s) ds -(1+\eps)(e^{C(t-t_0)}f(t_0)-f(t_0))\\
&\geq f(t)-e^{C(t-t_0)}f(t_0)-\eps(e^{C(t-t_0)}f(t_0)-f(t_0))\\
&\geq g_\eps(t).
\end{split}
\end{equation}
We argue by contradiction. If $g_\eps(t)\geq 0$ for some $t>t_0$, then let $t_1$ be the infimum of such $t$, we have $g_\eps(t_1)\geq 0$. Since $g_\eps(t_0)<0$, $t_1>t_0$ strictly, and for all $t\in[t_0,t_1)$, $g_\eps(t)<0$. This contradicts the fact that $C\int_{t_0}^{t_1}g_\eps(s)ds\geq g_\eps(t_1)$.
\end{proof}

\bibliography{bibfile}
\bibliographystyle{alpha}

\end{document}